\documentclass{article}
\usepackage{amsmath,amsthm}
\usepackage{amsfonts}
\usepackage{graphicx}
\usepackage{paralist}

\newtheorem{theorem}{Theorem}

\newtheorem{lemma}{Lemma}

 \makeatletter
    \renewenvironment{proof}[1][\proofname]{      \par\pushQED{\qed}\normalfont      \topsep6\p@\@plus6\p@\relax
      \trivlist\item[\hskip\labelsep\bfseries#1\@addpunct{.}]      \ignorespaces
    }{      \popQED\endtrivlist\@endpefalse
    }
    \makeatother
\input{tcilatex}
\begin{document}

\title{Rotors in triangles and tethrahedra}
\author{ Javier Bracho and Luis Montejano}

\maketitle

\section{Introduction}

A polytope $P$ is {\it circumscribed} about a convex body $\phi\subset \mathbb{R}^n$ if $\phi\subset P$ and each facet of $P$ is contained in a support hyperplane of $\phi$.  
We say that a convex body $\phi\subset \mathbb{R}^n$ is a {\it rotor} of a polytope $P$ if for each rotation $\rho$ of $\mathbb{R}^n$ there exist a translation $\tau$ so that $P$ is circumscribed about $\tau\rho\phi$.

If $Q^n$ is the $n$-dimensional cube then a convex body $\Phi$ is a rotor of $Q^n$ if and only if $\Phi$ has constant width. However, there are convex polytopes that have rotors which are not of constant width. 

A survey of results in this area has been given by Golberg \cite{Gol}. See also the book Convex Figures of Boltyanskii and Yaglom \cite{BY}.

It is well known that if $\Phi$ is a convex plane figure which is a rotor in the polygon $P$, then every support line of $\Phi$ intersects its boundary in exactly one point, and if $\Phi$ intersects each side of $P$ at the points $\{x_1, \dots x_n\}$, then the normals of $\Phi$ at these points  are concurrent.

In this paper we shall prove that if $P$ is a triangle, then there is a baricentric formula that describes the curvature of bd$\Phi$ at the contact points.
We prove also that if $\Phi\subset \mathbb{R}^3$ is a convex body which is a rotor in a tetrahedron  $T$ 
then the normal lines of $\Phi$ at the contact points with $T$ 
generically belong to one ruling of a quadric surface. 
\bigskip

\section{Rotors in the triangle}

Consider $\Phi$ a smooth  rotor in the triangle $T$ and suppose that the three sides of $T$ intersect the boundary of $\Phi$ at the points ${x_1, x_2,x_3}$, respectively. As in the case of constant width bodies in which 
the radii of curvature of the boundary  at the ends of a binormal sum to $h$, we are interested in a formula that involves the curvatures of the boundary of $\Phi$ at ${x_1, x_2,x_3}$.

A $C^m$ framed curve $(\alpha, \lambda)$ is a curve of class $C^m$ given by a parametrization of the following form:
there is a support function ${\cal P}:(-\delta,\delta)\to \mathbb{R}$ of class
 $C^m,$ $m\geq2$, such that $\alpha(\theta)= {\cal P}(\theta)u(\theta_0+\theta) + {\cal P}^\prime(\theta)u^\prime(\theta_0+\theta)$
 and  $\lambda$ is the tangent line through $\alpha(0)=x$, in the direction $x^\perp$.  Therefore, ${\cal P}^\prime(0)=0$ and 
 $\alpha(0)={\cal P}(0)u(\theta_0)$ is the closest point of the line $\lambda$ to the origin and the normal line of $\alpha$ at $\alpha(0)$ passes through the origin. 
 Where $u(\theta)=(\cos\theta,\sin\theta)$ and 
 $u^\prime(\theta)=(-\sin\theta,\cos\theta)$, for every $\theta\in \mathbb{R}.$
 
 A {\it sliding} along two given $C^n$ framed curves $(\alpha_1,\lambda_1)$ and $(\alpha_2,\lambda_2)$ is a one parameter family of Euclidean isometries $L_\theta$, $\theta\in (-\epsilon,\epsilon), \epsilon >0$, satisfying
 
 \begin{itemize}
 \item $L_0$ is the identity map,
 \item $L_\theta$ rotates the plane by an angle of $\theta$,
 \item $L_\theta (\lambda_i)$ is a tangent line of the curve $\alpha_i$, for each $\theta \in (-\epsilon,\epsilon)$ and $i=1,2$.
 \end{itemize} 

\begin{lemma} 
Let  $(\alpha_1,\lambda_1)$ and $(\alpha_2,\lambda_2)$ be two $C^n$ framed curves. Suppose that their normal lines at 
$\alpha_1(0)=x_1$ and $\alpha_2(0)=x_2$ are not parallel and are concurrent at the origin. Then
\begin{enumerate}
\item there is a unique sliding $L_\theta$, $\theta\in (-\epsilon,\epsilon), \epsilon >0$, along them, 
\item there is a $C^n$ map $f: (-\epsilon,\epsilon)\to \mathbb{R}^2$ such that $L_\theta( x)=R_\theta(x) + f(\theta)$, for every $x\in \mathbb{R}^2$, $f(0)=f^\prime(0)=0$,  where $R_\theta$ is the rotation of the plane about the origin by an angle of $\theta$.
\item If the origin does not lie in the line $\lambda_3$, then the envelope of $\{L_\theta(\lambda_3)\}_{\theta \in (-\epsilon,\epsilon)}$ is a  $C^n$ framed curve $(\alpha_3,\lambda_3)$, such that the tangent line at $\alpha_3(0)$ is $\lambda_3$ and the normal line at $\alpha_3(0)$ passes through the origin.
\end{enumerate}
\end{lemma}
\begin{proof} Let ${\cal E}$ be the Lie Group of orientation-preserving isometries of the Euclidean space $\mathbb{R}^2$. Let $R_\theta$ denote the rotation about the origin by an angle of $\theta$. Since every $g\in {\cal E}$ takes the form $g(x)=R_\theta(x) + f$ for some $\theta$ and a fixed $f\in \mathbb{R}^2$, we will identify a neighborhood of the identity in ${\cal E}$ with $(-\gamma, \gamma)\times \mathbb{R}^2\subset \mathbb{R}^3$, via the mapping $(\theta,f)\to R_\theta +f$.  Observe that the identity in ${\cal E}$ is identified with the origin in $\mathbb{R}^3$. 

Given a $C^m$ framed curve $(\alpha, \lambda)$ with support function ${\cal P}(\theta)$, consider the set
$$S=\{g\in {\cal E}\mid g(\lambda)\quad  \mbox{is a tangent line to} \quad \alpha \}$$
defined in the neighborhood of the identity in ${\cal E}$ (or of the origin in $\mathbb{R}^3$). We shall prove that $S$ is a surface of class $C^m$. Indeed, we have the following  explicit parametrization:  consider the map $\psi: \mathbb{R}^2\to \mathbb{R}^3$ given by 
$\psi (\theta,t)=(\theta, h(\theta, t))$, where $h(\theta, t)=({\cal P}(\theta)-{\cal P}(0))u(\theta_0+\theta) + tu^\prime(\theta_0+\theta)$, 
It is not difficult to verify that the for every $-\delta\leq\theta\leq\delta$ and $t\in R$, the isometry $L_\theta+ h(\theta, t)$ sends the line $\lambda$ to a tangent line of $\alpha$. Furthermore,
$$ \frac{d\psi}{d\theta}(0)=(1, {\cal P}^\prime(0)u(\theta_0))=(1,0,0)$$
and 
$$ \frac{d\psi}{dt}(0)=(0, u^\prime(\theta_0)).$$

Moreover, it follows that the normal vector to $S$ at the origin is $(0,-u(\theta_0))$.

Now, given two $C^m$ framed curves, $(\alpha_1,\lambda_1)$ and $(\alpha_2,\lambda_2)$, Let $S_1$ and $S_2$ be their corresponding surfaces. If $\alpha_i(0)={\cal P}_i(0)u(\theta_i)$, then the  normal vector to $S_i$ at the origin is $(0,-u(\theta_i))$, $i=1,2$, and since 
$\theta_1\not=\theta_2$, we have that in a neighborhood of the origin  $S_1$ and $S_2$  intersect transversally  in a curve of the form ($\theta, f(\theta))$ and hence the sliding can be written as

$$L_\theta= R_\theta + f(\theta)$$
where $f: (-\epsilon,\epsilon)\to \mathbb{R}^2$ is of class $C^m$.

Thus, for $i=1,2$ the support function of $\alpha_i$ is given by 
$${\cal P}_i(\theta)={\cal P}_i(0)+\langle f(\theta), u(\theta_i +\theta)\rangle .$$
where $\langle\cdot, \cdot \rangle $ denotes the interior product.

This implies that $f(0)=0$ and furthermore,  $0={\cal P}_i^\prime(0)=\langle f^\prime(0), u(\theta_i)\rangle $.  Since 
$\theta_1\not=\theta_2$, then $f^\prime (0)=0$.

Finally, let $\theta_3$  be such that $u(\theta_3)$  is orthogonal to the line $\lambda_3$ and let $r_3$ be the distance from $\lambda_3$ to the origin. Then  the support function of $\alpha_3$ is given by  
${\cal P}_3(\theta)=r_3+\langle f(\theta), u(\theta_3+\theta)\rangle $  and ${\cal P}_3^\prime(0)=0$ as we wished.
\end{proof}

\medskip

For curves of constant width $h$, the sum of the radii of curvature at extreme points of every diameter is $h$. For rotors in a triangle, the analogous result is the following baricentric formula. 

\medskip 
\begin{theorem} 

Let $\Phi$ be a rotor in the triangle $T$ with vertices $\{A_1,A_2,A_3\}$. Suppose the boundary of $\Phi$ is twice continuous differentiable and let $x_3=\Phi\cap A_1A_2$, $x_1=\Phi\cap A_2A_3$ and $x_2=\Phi\cap A_3A_1$.  
Let $\{a_1,a_2,a_3\}$ be the baricentric coordinates of the point $O$ with respect to the triangle $T$, where $O$ is the point at which the normal lines to $T$ at the points $x_1$, $x_2$ and $x_3$ concur. If $r_i$ is the distance from $O$ to $x_i$ and $\kappa_i$ the curvature of the boundary of $\Phi$ at $x_i$, $i=1,2,3$, then 
$$ \frac{ a_1}{\kappa_1 r_1} +  \frac{ a_2}{\kappa_2 r_2} +  \frac{ a_3}{\kappa_3 r_3}=1.$$

\end{theorem}

\begin{center}
\includegraphics[width=3.2in]{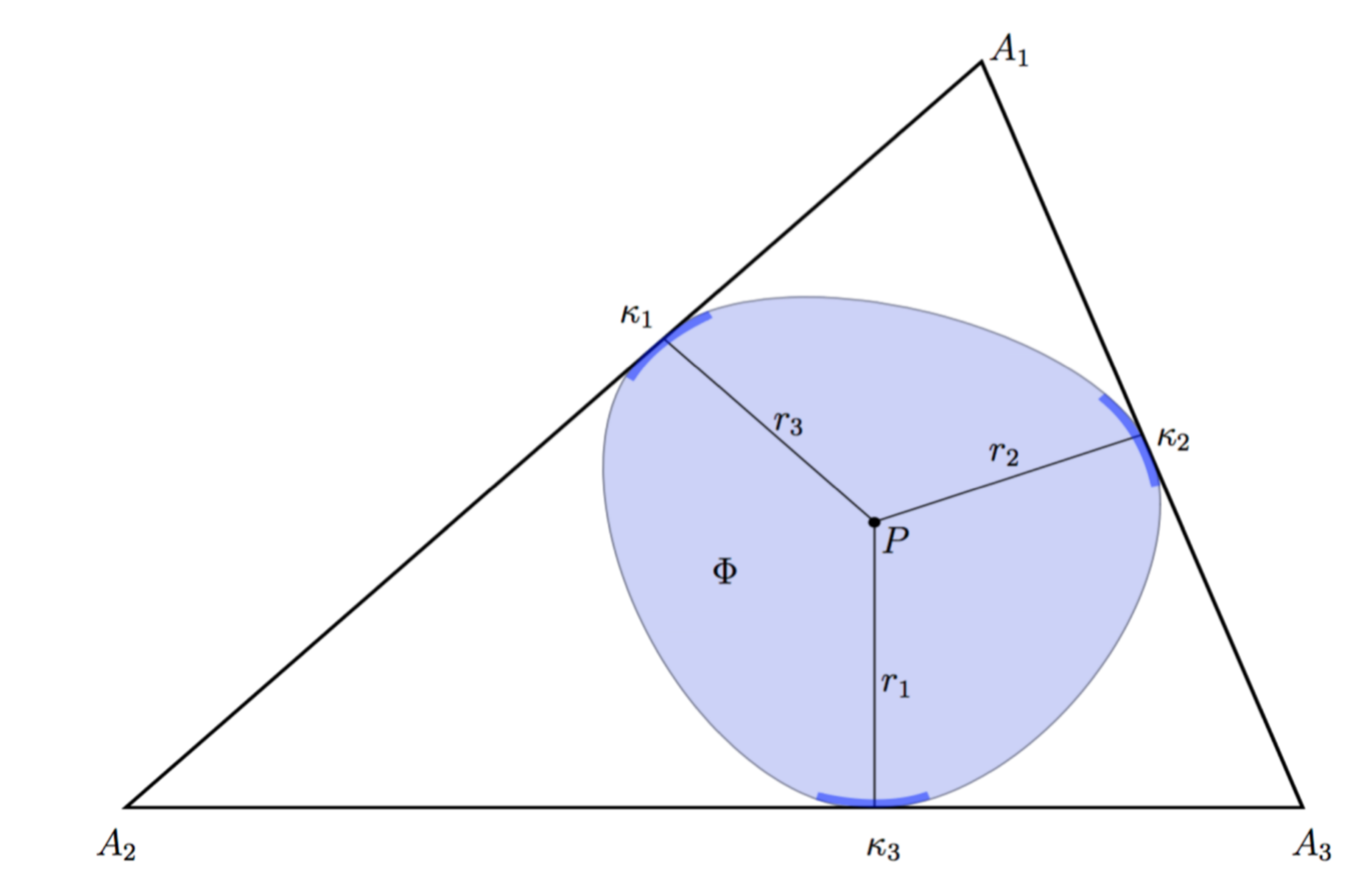} 

Figure 1  
\end{center}

\begin{proof}  Let $\alpha_i:(-\epsilon,\epsilon)\to \mathbb{R}^2$ be a $C^2$-parametrization of a neighborhood of the boundary of $\Phi$ around $x_i$,  with $\alpha_i(0)=x_i$ and let $\lambda_i$ be the line through $A_{i+1}A_{i+2}$, mod 3, so that $(\alpha_i, \lambda_i)$ are $C^2$ framed curves, whose corresponding normal lines at $x_i$ are concurrent at $O$. Suppose without loss of generality that $O$ is the origin. By Lemma 1, there is a sliding along the three framed curves. That is, there is a one parameter family of Euclidean isometries $L_\theta$, $\theta\in (-\epsilon,\epsilon), \epsilon >0$, satisfying
 
 \begin{itemize}
 \item $L_0$ is the identity map,
 \item $L_\theta$ rotates the plane by an angle of $\theta$,
 \item $L_\theta (\lambda_i)$ is a tangent line of the curve $\alpha_i$, for each $\theta \in (-\epsilon,\epsilon)$ and $i=1,2, 3$.
 \end{itemize} 
 
 Furthermore, there is a $C^2$ map $f: (-\epsilon,\epsilon)\to \mathbb{R}^2$ such that 
 $$L_\theta( x)=R_\theta(x) + f(\theta),$$
  for every $x\in \mathbb{R}^2$, $f(0)=f^\prime(0)=0$,  where $R_\theta$ is the rotation of the plane through the origin by an angle of $\theta$.
  
  Let ${\cal P}_i(\theta)$ be the pedal function of the framed curve $\alpha_i$, with ${\cal P}_i(0)=r_i= \break \mid x_i \mid$, $i=1,2,3$. Hence, ${\cal P}_i^\prime(0)=0$ and the radius of curvature of the boundary of $\Phi$ at $x_i$ is 
  $$\frac{1}{\kappa_i}= {\cal P}_i(0)+{\cal P}_i^{\prime\prime}(0).$$
  
  On the other hand, ${\cal P}_i(\theta) = \mid L_\theta( x_i)\mid =\mid R_\theta(x_i) + f(\theta)\mid.$   Hence,
  $${\cal P}_i(\theta)^2= \langle R_\theta(x_i) + f(\theta), R_\theta(x_i) + f(\theta)\rangle .$$

So,

$${\cal P}_i(\theta) {\cal P}_i^\prime(\theta)= \langle R_\theta(x_i) + f(\theta), R_\theta(x_i)^\perp + f^\prime(\theta)\rangle.$$

Let $h_i(\theta)= \langle R_\theta(x_i), f^\prime(\theta)\rangle  +\langle R_\theta(x_i)^\perp, f(\theta)\rangle + \langle f^\prime(\theta), f^\prime(\theta)\rangle $ in such a way that 
$${\cal P}_i^\prime(\theta)=\frac{h_i(\theta)}{{\cal P}_i(\theta)}$$

and

$${\cal P}_i^{\prime\prime}(\theta)=\frac{h_i^\prime(\theta){\cal P}_i(\theta)^2-h_i(\theta)^2}{{\cal P}_i(\theta)^3}.$$

Note that  $h_i(0)=0$ and $h_i^\prime(0)=\langle f^{\prime\prime}(0), x_\rangle $.

Since the radius of curvature of bd$\Phi$ at $x_i$ is given by ${\cal P}_i(0)+{\cal P}_i^{\prime\prime}(0)$, we have that for $i=1,2,3$

$$\frac{1}{\kappa_i}=r_i + \frac{\langle f^{\prime\prime}(0), x_i\rangle }{r_i}.$$

Let $\{b_1,b_2,b_3\}$ be the baricentric coordinates of the origin $O$ with respect the triangle with vertices $\{x_1,x_2,x_3\}$. That is: $b_1x_1+  b_2x_2+ b_3x_3=0$, with $b_1+b_2+b_3=1$. Hence, for $i=1,2,3$,

$$\frac{ b_ir_i^2}{\kappa_ir_i}=b_ir_i^2 + \langle f^{\prime\prime}(0), b_ix_i\rangle ,$$
and therefore,
$$\sum\frac{ b_ir_i^2}{\kappa_ir_i}=\sum b_ir_i^2+0.$$

To conclude the proof of the theorem, it will be enough to prove that 
$$ a_i = \frac{b_ir_i^2}{b_1r_1^2+b_2r_2^2+b_3r_3^2}.$$

The basic property that defines $A_i$ is $\langle A_i,x_j\rangle =\langle x_j,x_j\rangle =r_j^2$ for $i\not=j$.  Using it, one easily obtains that 
$$\langle b_1r_1^2A_1+b_2r_2^2A_2+b_3r_3^2A_3, x_j\rangle =\langle r_j^2A_j, b_1x_1+  b_2x_2+ b_3x_3\rangle =0,$$
for $j=1,2,3.$  This implies that $b_1r_1^2A_1+b_2r_2^2A_2+b_3r_3^2A_3=0$ because the $x_j$ generate $\mathbb{R}^2$, and from here 

$$\frac{b_1r_1^2}{\sum b_ir_i^2}A_1+\frac{b_2r_2^2}{\sum b_ir_i^2}A_2+\frac{b_3r_3^2}{\sum b_ir_i^2}A_3=0.$$

It follows that 
$$\frac{a_1}{\kappa_1r_1}+\frac{a_2}{\kappa_2r_2}+\frac{a_3}{\kappa_3r_3} = 1,$$
as we wished. 

\end{proof}

\section { The relation with immobilization problems}

Immobilization problems were introduced by Kuperberg \cite{K} and also appeared in \cite {O}. They were motivated by grasping problems in robotics (\cite {MNP1} and \cite {MNP2}). 

Let $\Phi \subset \mathbb{R}^n$ be a convex body. A collection of points $X$ on the boundary of $\Phi$ is said to immobilize $\Phi$ if any small rigid movement of $\Phi$ causes one point in $X$ to penetrate the interior of $\Phi$.  In the plane, for the case in which three points $X=\{x_1,x_2, x_3\}$ lie in the boundary $\Phi$, there is a baricentric formula involving the curvature of bd$\Phi$ at $x_i$ that allows us to know if $X$ immobilizes $\Phi$.  See \cite{BM}.

\begin{theorem} 

Let $\Phi$ be a twice continuous differentiable convex figure and let   $X=\{x_1,x_2, x_3\}$ be three points in the boundary of $\Phi$, whose normals are concurrent at the point $O$. 
Let $\{a_1,a_2,a_3\}$ be the baricentric coordinates of the point $O$ with respect to the vertices of the triangle formed be the three support lines of $\Phi$ at $x_1$, $x_2$ and $x_3$. Also, let $r_i$ be the distance from $O$ to $x_i$, let $\kappa_i$ be the curvature of the boundary of $\Phi$ at $x_i$, $i=1,2,3$, and let 
$$ \omega=  a_1\kappa_1 r_1 +  a_2\kappa_2 r_2+   a_3\kappa_3 r_3.$$
Then, if $\omega <1$, $\{x_1,x_2, x_3\}$ immobilize $\Phi$, and if $\omega >1$,  they do not. 

\end{theorem}


There is a duality between Theorem~2 and Theorem~1.  While in Theorem2, we have a rigid segment sliding along the boundary of the convex figure $\Phi$, in Theorem~1, we have a rigid angle (formed by two lines) sliding along the boundary of $\Phi$.

In dimension three, immobilization results are much more complicated. See \cite {BFMM}. To characterize when four points in the faces of a tetrahedron $T$  immobilize $T$ we require the following definition. 

Let $\{L_1, L_2, L_3, L_4\}$ be four  directionally independent lines in $\mathbb{R}^3$. We say that they belong {\it generically to one ruling of a quadric surface} if 
\begin{itemize}
\item they are concurrent,
\item they belong to one ruling of a quadric surface, or
\item they meet in pairs and the planes these pairs generate meet in the line through the intersecting points.
\end{itemize}

\begin{theorem}
A necessary and sufficient condition for four points $\{x_1, x_2, x_3, x_4\}$, in the corresponding four faces of a tetrahedron $T$, to immobilize it,  is that the normal lines to $T$ at $x_1, x_2, x_3$ and $x_4$  belong generically to one ruling of a quadratic surface. 
\end{theorem}

The ``duality" mentioned above, gives us the following theorem for rotors in a tetrahedron. 

\begin{theorem} 

Let $\Phi$ a twice continuous differentiable rotor in the tetrahedron $T$, and let $\{x_1, x_2, x_3, x_4\}$ be the points of the boundary of $\Phi$ that intersect the four faces of $T$. Then, the normal lines to $T$ at $x_1, x_2, x_3$ and $x_4$  belong generically to one ruling of a quadratic surface. 
\end{theorem}

\begin {proof} Consider a tetrahedron $T$ that circumscribes $\Phi$.  For every $\rho \in SO(3)$, let $T(\rho)$ be the tetrahedron directly homothehtic to $\rho T$ circumscribing $\Phi$ and let $V_{\Phi}(\rho)$ be the volume of  of $T(\rho)$. It is not difficult to see that $V_{\Phi}(\rho)$ depends continuously on $\rho$.

We will prove that if $\rho_0$ is a local maximum of $V_{\Phi}(\rho)$, then 
the four normal lines to the boundary of $\Phi$  at the points that touch the four faces of $T(\rho_0)$, belong generically to one ruling of a quadratic surface.  If this is so, then the proof the theorem is complete because $\Phi$ is a rotor in 
$T$ if and only if $V_{\Phi}(\rho)$ is constant. 
For the proof of the above statement, it will be sufficient  to consider the case in which $\Phi$ is a tetrahedron. The reason is that if $a,b,c$  and $d$ are the points in which the sides of $T(\rho_0)$ touch the boundary of $\Phi$, then 
$\rho_0$ is also a local maximum of $V_{K}(\rho)$, where $K$ is the tetrahedron with vertices $\{a,b, c,d\}$.

Let $H_a, H_b, H_c$ and $H_d$ be four planes containing the  faces of the tetrahedron  $T(\rho_0)$, in such a way that $a\in H_a, b\in H_b, c\in H_c$ and $d\in H_d$, respectively. Assume now that a $T(\rho_0)$ is a rigid tetrahedron  sliding along $a, b, c$.  That is, $T(\rho_0)$ is sliding rigidly in such a way that the points $a, b, c$  remain fixed but inside the planes  $H_a, H_b$ and $H_c$, and during the rigid sliding movement of $T(\rho_0)$, the fixed point $d$ is  always inside $T(\rho_0)$.
 
The proof of Theorem 4 now follows straightforward from the proof or Theorem 3 in \cite{BFMM}, but this time we consider, instead of a rigid triangle sliding along three fixed planes, the dual situation of a $3$-dimensional rigid sector (the angle between three planes $H_a, H_b$ and $H_c$) sliding along three fixed points $a,b,c$. 

\end{proof}


\begin{thebibliography}{11}


\bibitem{BM} Bracho J.,  Montejano L. and Urrutia J. Immobilization of smooth convex curves. \emph{Geometriae Dedicata}. \textbf{53} (1994), 119-131.
\bibitem{BFMM} Bracho J.,  Fetter H., Mayer D. and Montejano L. Immobilization of solids and mondriga quadratic forms. \emph{Journal of the London Math. Soc.} \textbf{51} (1995), 189-200.
\bibitem{BY}  Boltianski, W.G.\ and Yaglom, I.M., \emph{Convex Figures.} Holt Rinehart  and Winston, New York 1961.
\bibitem{Gol} Golberg, M.,  Rotors in polygons and polyhedra. \emph{Math. Comput.} \textbf{14}, (1960), 229-239.


\bibitem{K} Kuperberg W.,  DIMACS Workshop in Polytopes \emph{Rutgers University . Jan}. 1990.

\bibitem{MNP1} Markenscoff X., Ni L. and Papadimitrou CH. H.,Optimal grid of a polygon. \emph{Int J. Robotics Research}. \textbf{8}(2) (1989), 17-29.

\bibitem{MNP2} Markenscoff X., Ni L. and Papadimitrou CH. H.,The geometry gof grasping. \emph{Int J. Robotics Research}. \textbf{9}(1) (1990), 61-74.

\bibitem{O} O'Rourke. J. Comptutational Geometry, column 9 \emph{SIGACT News}. \textbf{21}(1) (1990), 18-20, No.74.



\end{thebibliography}
\end{document}